\documentclass[a4paper,11pt]{amsart}

\usepackage{amsmath}
\usepackage{amssymb}
\usepackage{amsthm}
\usepackage[dvipdfmx]{graphicx}
\usepackage{amscd}
\usepackage{comment}
\usepackage{color}
\usepackage{url}  
\usepackage{mathrsfs} 

\newcommand{\supp}{\mathop{\mathrm{supp}}\nolimits}                       

\begin{document}

\numberwithin{equation}{section}

\theoremstyle{definition}
\newtheorem{theorem}{Theorem}[section]
\newtheorem{proposition}[theorem]{Proposition}
\newtheorem{formula}{Formula}[section]
\newtheorem{lemma}[theorem]{Lemma}
\newtheorem{corollary}[theorem]{Corollary}

\newtheorem{remark}{Remark}[section]
\newtheorem{definition}{Definition}[section]
\newtheorem{example}{Example}[section]
\newtheorem{problem}{Problem}[section]
\newtheorem{convention}{Convention}[section]
\newtheorem{conjecture}{Conjecture}[section]

\title[Finitely generated simple groups of homeomorphisms of $S^1$]{On a family of finitely generated simple groups of homeomorphisms of the circle}
\author[M. Kato]{Motoko Kato}
\thanks{The author is supported by JSPS KAKENHI Grant number 20K14311, 25K17256 and JST, ACT-X Grant Number JPMJAX200A, Japan.}
\address[M. Kato]{Faculty of Education, University of the Ryukyus, 1 Sembaru, Nishihara-Cho, Nakagami-Gun, Okinawa, 903-0213 Japan}
\email{katom@cs.u-ryukyu.ac.jp}  

\maketitle

\begin{abstract}
The notion of chain groups of homeomorphisms of $\mathbb{R}$ was introduced by Kim, Koberda and Lodha as a generalization of Thompson's group $F$. Subsequently, 
an $S^1$-version of chain groups, known as ring groups, has been studied.
In this paper, we further study the simplicity of the commutator subgroups of ring groups. 
We show that a ring group with a prechain subgroup acting minimally on its support has a simple commutator subgroup.
We also study isometric actions of ring groups on $\mathbb{R}$-trees.
We give a construction of ring groups such that for every fixed point-free isometric action on an $\mathbb{R}$-tree, there exists an invariant line upon which the group acts by translations.
In other words, such ring groups have property A$\mathbb{R}$. 
We also confirm that there exist uncountably many finitely generated simple groups in the group of orientation preserving homeomorphisms of $S^1$,
which are commutator subgroups of ring groups.

\end{abstract}

\section{Introduction}
{\it Thompson's groups} $F$ and $T$ are finitely presented groups consisting of homeomorphisms of the unit interval $[0,1]$ and the circle $S^1$, respectively.
These groups were originally studied by Richard Thompson in the 1960s, and are known to have many interesting properties.
There exist various generalizations of Thompson's groups, such as the Brown-Higman-Thompson groups $F_n$ and $T_n$ (\cite{Brown}).
Among these generalizations, chain groups of homeomorphisms of $\mathbb{R}$ were introduced by Kim, Koberda and Lodha (\cite{KKL}).
A {\it prechain group} is generated by a finite sequence of orientation preserving homeomorphisms of $\mathbb{R}$ whose supports form a ``chain'' of open intervals.
A prechain group is called a {\it chain group} if subgroups generated by two of the generators are isomorphic to $F$ if the supports have nonempty intersection.
In particular, every $2$-chain group is isomorphic to $F$.
The Brown-Higman-Thompson groups $F_n$ are chain groups. In fact, by taking sufficiently large $N$, the subgroup of an $n$-chain group
generated by $N$-th powers of generators is isomorphic to $F_n$ (\cite{KKL}).
An ``$S^1$-version'' of chain groups has been previously considered (\cite{BBK}, \cite{HLR}). 
A {\it prering group} is generated by a sequence of orientation preserving homeomorphisms of $S^1$ whose supports form a ``ring'' of three or more open intervals that cover the entire $S^1$.
A prering group is called a {\it ring group} if subgroups generated by two of the generators are isomorphic to $F$ when their supports have nonempty intersection.
Taking generators with consecutive indices, we may define a {\it prechain subgroup} (resp.\ {\it chain subgroup}) of a prering group (resp.\ ring group).
It is known that $T_n$ are ring groups (\cite{Kato2023}), but the diversity of the ring groups compared to $T_n$ has not been investigated.
In particular, we further study some properties of commutator subgroups of ring groups to find similarities between ring groups and $T_n$.
We are particularly interested in a property of groups called property A$\mathbb{R}$, which obstructs actions on real trees. This property is used to show that the commutator subgroups are finitely generated.
First, we give a sufficient condition for a group to have a simple commutator subgroup. 
This theorem is not limited to ring groups, but can be applied to more general situations.
\begin{theorem}\label{Higman1_lem}
Let $G$ be a group acting faithfully on a set $X$.
Let $S$ be a generating set of $G$.
Suppose that 
every $s_1, s_2\in S\setminus \{1\}$ and every $g\in G\setminus \{1\}$ satisfy
either 
there exists $u\in G$ such that 
\begin{align}\label{Higman1_eq}
\supp(s_1)\cap (u g u^{-1})\left(\supp(s_2)\right)=\emptyset,
\end{align}
or
there exist $u_1, u_2\in G$ such that 
\begin{align}\label{Higman2_eq}
\supp(s_1) \cap (u_2 g u_2^{-1})(u_1g u_1^{-1})\left(\supp(s_2)\right)=\emptyset.
\end{align}
Then for every normal subgroup $N\neq 1$ of $G$, $N$ contains the commutator subgroup $[G,G]$ of $G$.
\end{theorem}
\noindent Theorem~\ref{Higman1_lem} is a variation of Higman's theorem (\cite[Theorem 1]{Higman}).
Higman's theorem states that for a group $G$ acting faithfully on a set $X$, if for every $g_1, g_2, g_3\in G\setminus \{1\}$,
there exists $u\in G$ satisfying 
\begin{align*}
\left( \supp(g_1)\cup \supp(g_2)\right) \cap (u g_3 u^{-1})\left( \supp(g_1)\cup \supp(g_2)\right)=\emptyset,
\end{align*}
then $[G,G]$ is simple.

Second, using Theorem~\ref{Higman1_lem}, we show the simplicity of commutator subgroups of ring groups under an assumption on minimality of a chain subgroup. 
\begin{theorem}\label{simple_thm}
Let $G$ be a ring group with a prechain subgroup whose action on its support is minimal.
Then the commutator subgroup $[G,G]$ of $G$ is simple.
\end{theorem}
\noindent A similar result is in \cite[Theorem 3.2.2]{HLR}, which demonstrates the simplicity of the commutator subgroup when the action of the ring group on $S^1$ is minimal.

Third, we discuss isometric actions of ring groups on $\mathbb{R}$-trees. 
\begin{theorem}\label{AR_thm}
Let $H$ be a $2$-generated subgroup of $\mathrm{Homeo}^+(\mathbb{R})$.
There exists a ring group $G$ such that 
\begin{itemize}
\item[(i)] $H$ embeds into a chain subgroup of $G$ which acts minimally on its support, and
\item[(ii)] for every non-trivial action of $G$ on an $\mathbb{R}$-tree, there exists a line in the $\mathbb{R}$-tree upon which the group acts by translations.
That is, $G$ has property A$\mathbb{R}$.
\end{itemize}
\end{theorem}
\noindent 
In the proof of Theorem~\ref{AR_thm}, (i) follows from a result on embeddings of $n$-generated subgroups of $\mathrm{Homeo}^+(\mathbb{R})$ into chain groups (\cite{KKL}). 
The main part of the proof is to verify (ii): $G$ has property A$\mathbb{R}$. We use a result by Culler and Vogtmann.
They gave a group-theoretic criterion for a group to have property A$\mathbb{R}$ as conditions on a set of generators of the group (\cite{CV}). 

Finally, we have the following corollary. 
\begin{corollary}\label{uncountable2_cor}
There exist uncountably many isomorphism types of commutator subgroups of the ring groups given in Theorem~\ref{AR_thm}.
In particular, there exists a family of finitely generated simple subgroups of $\mathrm{Homeo}^+(S^1)$ containing uncountably many isomorphism types of groups.
\end{corollary}
\noindent We note that the existence of uncountably many finitely generated simple subgroups of $\mathrm{Homeo}^+(S^1)$ is not a new result.
Indeed, it is known that there exist uncountably many finitely generated simple subgroups of $\mathrm{Homeo}^+(\mathbb{R})$ (\cite{HL}). Theorem~\ref{AR_thm} provides an alternative proof to this fact.

The rest of this paper is organized as follows.
In Section~\ref{group_sec}, we recall precise definitions of $F$, $T$, $F_n$, $T_n$, 
chain groups and ring groups.
In Section~\ref{simple_sec}, we show Theorems~\ref{Higman1_lem} and \ref{simple_thm}.
In Section~\ref{fixed_sec}, we recall the criterion for property A$\mathbb{R}$ in \cite{CV}, and show Theorem~\ref{AR_thm}.
Finally, in Section~\ref{uncountable_sec}, 
we show Corollary~\ref{uncountable2_cor}.

\section{Thompson's groups, chain groups and ring groups}\label{group_sec}
In this section, we recall definitions and properties of Thompson's groups, chain groups and ring groups.
We first gather some notations and conventions.
Let $M$ be either $\mathbb{R}$ or $S^1$.
We write $\mathrm{Homeo}^+(M)$ for the group of orientation preserving homeomorphisms of $M$.
For $f\in \mathrm{Homeo}^+(M)$, we write $\supp(f)$ for the set-theoretic support of $f$:
$$\supp(f)=\{t\in M\mid f(t)\neq t\}.$$
For $S\subset \mathrm{Homeo}^{+}(M)$, we define $\supp(S)$ by
$$\supp(S):= \bigcup_{s\in S} \supp(s).$$
We note that $\supp(f)$ and $\supp(S)$ are open subsets of $M$.
We fix an orientation on $S^1$ so that we may consider open intervals between two points of $S^1$.
For $a, b\in M$ and an open interval $J=(a,b)$ in $M$, let
$$\partial_{-}J:= a,\ \partial_{+}J:= b.$$
For a group $G$, we denote the commutator subgroup by $[G,G]$. 
For $f$, $g\in G$, we define $[f,g]$ by 
$$[f,g]:= fgf^{-1}g^{-1}.$$

We start with definitions of {\it Thompson's groups} $F$ and $T$, following \cite{CFP}. 
$F$ is the group of piecewise affine, orientation preserving homeomorphisms of the unit interval that are differentiable except at finitely many dyadic rational numbers, and
with derivatives of powers of $2$ on intervals of differentiability.

We consider $S^1$ with a basepoint as the interval with the endpoints identified.
Then $T$ is the group of piecewise affine, orientation preserving homeomorphisms of $S^1$ that preserve the set of dyadic rational numbers and that are differentiable except at finitely many dyadic rational numbers and with derivatives of powers of $2$ on intervals of differentiability.
$F$ is identified with the subgroup of $T$ consisting of elements that fix the basepoint of $S^1$.

Next, we recall definitions and properties of chain and ring groups.
Let $n \geq 2$ be a natural number.
An {\it $n$-chain of intervals} is a sequence of $n$ nonempty open intervals $(J_1,\ldots, J_n)$ in $\mathbb{R}$, satisfying that 
\begin{itemize}
\item[(C1)] $J_i\cap J_j=\emptyset$ if $|i-j| >1$,
\item[(C2)] $J_{i}\cap J_{i+1}$ is a proper nonempty subinterval of $J_{i}$ and $J_{i+1}$ for every $i\in \{1,\ldots, n-1\}$.
\end{itemize}
Let $\mathcal{F}=\{f_i\}_{1\leq i\leq n}$ be a set of orientation preserving homeomorphisms of $\mathbb{R}$ such that the support of every $f_i$ is a nonempty open interval and the sequence $(\supp(f_1),\ldots, \supp(f_n))$ is an $n$-chain of intervals.
An {\it $n$-prechain group} $G_{\mathcal{F}}$ is a group 
generated by 
such an $\mathcal{F}$.
An {\it $n$-chain group} is an $n$-prechain group such that $\langle f_i, f_{i+1}\rangle$ is isomorphic to Thompson's group $F$ for every $1\leq i\leq n-1$. 
A group is a {\it chain group} if it is an $n$-chain group for some $n$.
We note that the support of a prechain group or a chain group is contained in some interval of $\mathbb{R}$.
The following lemma shows that the definition of a chain group is natural in a certain sense.
\begin{lemma}[{\cite[Lemma 3.1]{KKL}}]\label{2chain_lem}
Let $G_{\mathcal{F}}$ be a $2$-prechain group with respect to $\mathcal{F}=\{f_1, f_2\}$. 
Suppose that 
\begin{align}\label{strong_condition}
f_2f_1(\partial_{-}\supp(f_2))\geq \partial_{+}\supp(f_1).
\end{align}
Then $f_1$ and $f_2$ satisfy
$[f_1, \left(f_2f_1\right)f_2\left(f_2f_1\right)^{-1}]=1$
and the group generated by $f_1$ and $f_2$ is isomorphic to Thompson's group $F$. 
\end{lemma}

In this paper, we study an ``$S^1$-version'' of chain groups studied in \cite{HLR} and \cite{Kato2023}. 
Let $m\geq 3$ be an integer.
An {\it $m$-ring of intervals} is a sequence of $m$ nonempty open intervals  $(J_1,\ldots, J_m)$ in $S^1$, satisfying that 
\begin{itemize}
\item[(R1)] $J_i\cap J_j=\emptyset$ if $ 1<|i-j|<m-1$, and
\item[(R2)] $J_{i}\cap J_{i+1}$ is a proper nonempty subinterval of $J_{i}$ and $J_{i+1}$ for every $1\leq i\leq m$, where indices are considered modulo $m$.
\end{itemize}
Note that, by definition, the union of $J_1, \ldots, J_m$
covers the entire $S^1$.
Let $\mathcal{R}=\{r_i\}_{1\leq i\leq m}$ be a set of orientation preserving homeomorphisms of $S^1$ such that the support of every $r_i$ is a nonempty open interval and
the sequence $(\supp(r_1),\ldots, \supp(r_m))$ is an $m$-ring of open intervals.
An {\it $m$-prering group} $G_{\mathcal{R}}$ is a subgroup of $\mathrm{Homeo}^+(S^1)$
generated by such an $\mathcal{R}$.
An {\it $m$-ring group} is an $m$-prering group such that 
$\langle r_i, r_{i+1}\rangle$ and $\langle r_m, r_{1}\rangle$ is isomorphic to $F$ for every $i$.
A group is called a {\it ring group} if it is an $m$-prering group for some $m$.
For $2\leq n\leq m-1$, an {\it $n$-prechain subgroup} (resp.\ {\it $n$-chain subgroup}) of an $m$-prering group (resp.\ $m$-ring group) is a subgroup generated by $n$ elements of $\mathcal{R}$ 
with consecutive indices, where indices are considered modulo $m$.

\section{Main results}

\subsection{Simplicity of the commutator subgroups}\label{simple_sec}
In this section, we show Theorems~\ref{Higman1_lem} and \ref{simple_thm}.

First, we prove Theorem~\ref{Higman1_lem}.
\begin{proof}[Proof of Theorem~\ref{Higman1_lem}]
Let $N$ be a nontrivial normal subgroup of $G$.
Let $g\in N\setminus \{1\}$.
We show that $[s_1,s_2]\in N$ for every $s_1, s_2\in S$.

First, we suppose that there exists $u\in G$ such that $s_1, s_2, g, u$ satisfy $(\ref{Higman1_eq})$.
Let $g'=u g u^{-1}$. Since $g\in N$, which is a normal subgroup of $G$, $g'\in N$.
According to $(\ref{Higman1_eq})$,
$$\supp(s_1)\cap g'\supp(s_2)=\supp(s_1)\cap \supp(g' s_2 {g'}^{-1})=\emptyset.$$
It follows that $s_1$ and $g' s_2 {g'}^{-1}$ commute.
Therefore,
\begin{align*}
[s_1,s_2] &= s_1s_2s_1^{-1}s_2^{-1}= s_1 s_2 (g' s_2^{-1} {g'}^{-1} s_1^{-1} g' s_2 {g'}^{-1}) s_2^{-1} \\
&= \left(s_1 s_2 \left(g' (s_2^{-1} {g'}^{-1} s_2)\right) s_2^{-1} s_1^{-1}\right) g' (s_2 {g'}^{-1} s_2^{-1}).
\end{align*}
Since conjugates of $g'$ are in $N$, $[s_1,s_2]\in N$. 

Next, we suppose that there exist $u_1, u_2\in G$ such that $s_1, s_2, g, u_1, u_2$ satisfy $(\ref{Higman2_eq})$.
Let $g'=(u_2 g u_2^{-1})(u_1g u_1^{-1})$. 
By a similar argument, it follows that $[s_1,s_2]\in N$.

Since $[G,G]$ is normally generated by commutators of generators,
it follows that $N\supset [G,G]$.
\end{proof}

Next, we prepare some definitions and lemmas for Theorem~\ref{simple_thm}.
\begin{definition}[CO-transitive, locally CO-transitive, minimal]
Let $G$ be a group acting on a topological space $X$. 
We say the action is \textit{CO-transitive} (or \textit{compact-open transitive}) if for every proper compact subset $K$ of $X$ and every nonempty open subset $J$ of $X$, there exists an element $g$ in $G$ such that $g(K)\subset J$. 
We say the action is \textit{locally CO-transitive} if for every proper compact subset $K$ in $X$, there exists $x\in X$ such that for every nonempty open neighborhood $J$ of $x$, there exists an element $g$ in $G$ such that $g(K)\subset J$. 
We say the action is \textit{minimal} if every orbit is dense in $\supp(G)$. 
\end{definition}
\noindent Following \cite{KKL}, an action is called minimal if its orbits are required to be dense only in the support, and not necessarily in the entire space.

\begin{lemma}[{\cite[Lemma 3.6]{KKL}}]\label{minimal_chain_lem}
Let $G=G_{\mathcal{F}}$ be a prechain group. 
\begin{itemize}
\item[(1)] For each $g\in G$ and for each compact set $K$ in $\supp(G)$, there exists $u\in [G,G]$
such that the actions of $g$ and $u$ agree as functions on $K$.
In particular, $\supp(G)=\supp([G,G])$.
\item[(2)] The natural action of $[G,G]$ on $\supp(G)=\supp([G,G])$ is locally CO-transitive.
In fact, let $x$ be a boundary point of $\supp(f)$ of some $f\in \mathcal{F}$.
For every proper compact subset $K\subset \supp(G)$ and every open neighborhood $J$ of $x$, there exists an element $g$ in $[G,G]$
such that $g(K)\subset J$. 
\item[(3)] If the action of G on supp(G) is minimal, then the action of $[G,G]$ on $\supp(G)=\supp([G,G])$ is CO-transitive.
\end{itemize}
\end{lemma}

\begin{lemma}\label{for-minimal_ring_prop}
Let $m\geq 3$.
Let 
\begin{itemize}
\item $G=G_{\mathcal{R}}$ be an $m$-prering group, 
\item $G_{\mathcal{R}'}$ a prechain subgroup of $G_{\mathcal{R}}$, 
\item $H$ an $(m-1)$-prechain subgroup of $G_{\mathcal{R}}$, 
\item $K$ a compact subset of $\supp(H)$.
\end{itemize}
Then there exists $u\in [G,G]$ such that 
$$u(K)\subset \supp(G_{\mathcal{R'}}).$$
\end{lemma}

\begin{proof}
Since $H$ is an $(m-1)$-prechain subgroup, there exists $r\in \mathcal{R}$ and a boundary point $x$ of $\supp(r)$ such that $x\in \supp(G_{\mathcal{R'}})\cap \supp(H)$.
We fix an open neighborhood $J_x$ of $x$ in $\supp(G_{\mathcal{R'}})\cap \supp(H)$.
According to Lemma~\ref{minimal_chain_lem} (2), there exists $u\in [H,H]$ such that 
 $$u(K)\subset J_x\subset \supp(G_{\mathcal{R'}}).$$
\end{proof}

The following proposition can be proved in the same way as \cite[Lemma 3.2.1]{HLR}.
For convenience, we include a proof.
\begin{proposition}[{\cite[Lemma 3.2.1]{HLR}}]\label{minimal_ring_prop}
Let $m\geq 3$.
Let $G=G_{\mathcal{R}}$ be an $m$-prering group with respect to $\mathcal{R}=\{r_i\}_{1\leq i\leq m}$. 
If $G$ has a prechain subgroup whose action on its support is minimal, then the action of $[G,G]$ on $S^1$ is CO-transitive. 
\end{proposition}

\begin{proof}
We write $G_{\mathcal{R'}}$ for the prechain subgroup whose action on its support is minimal.
Here, $\mathcal{R'}\subset \mathcal{R}$ is the set of generators as a chain subgroup.
We fix a proper compact subset $K$ and a nonempty open subset $J$ of $S^1$.
Without loss of generality, we may assume the following:
\begin{itemize} 
\item $K$ is a closed interval in $S^1$ such that $\partial_{-}K\in \supp(r_1)$.
\item $J$ is an open interval, and there exists $1\leq j\leq m$ such that the closure $cl(J)$ of $J$ is contained in $\supp(r_j)$.
In fact, for
\begin{align}\label{epsilon_eq}
\varepsilon:= \frac{1}{2}\min_{1\leq i\leq m}|\partial_{-}\supp(r_{i+1})-\partial_{+}\supp(r_i)|
\end{align}
and 
$$J_i:= \left(\partial_{-}\supp(r_i)+\varepsilon, \partial_{+}\supp(r_i)-\varepsilon\right) \quad(1\leq i\leq m),$$
$\{J_i\}_{1\leq i\leq m}$ covers $S^1$ and there exists $j$ such that $J_j\cap J\neq \emptyset$. 
We replace $J$ with $J_j\cap J$.
We note that the closure of such $J$ is in $\supp(r_j)$.
\end{itemize}
First, we show that there exists $u_0\in [G,G]$ such that $u_0(cl(J))$ is included in $\supp(r_j)\cap \supp(r_{j+1})$, where $cl(J)$ denotes the closure of $J$.
We take sufficiently large $p\in \mathbb{N}$ so that $r_{j}^p(cl(J))$ is included in $\supp(r_j)\cap \supp(r_{j+1})$.
Applying Lemma~\ref{minimal_chain_lem} (1) to the prechain subgroup generated by $r_j$ and $r_{j+1}$,
we take $u_0\in [G,G]$ such that $u_0$ and $r_j^p$ agree as functions on the closure of $J$. 
Then $u_0(cl(J))$ is included in $\supp(r_j)\cap \supp(r_{j+1})$.

Next, we consider two cases: (1) the case where 
$$K\subset \left(\partial_{-}\supp(r_{1}), \partial_{+}\supp(r_{m-1})\right)$$
and (2) otherwise. 

In the case of (1), since we assumed that $\partial_{-}K\in \supp(r_1)$,
we have $\partial_{+}K\in \supp\langle r_1,\ldots, r_{m-1}\rangle$.
Therefore, $K$ is contained in the support of an $(m-1)$-prechain subgroup
$$H_1:= \langle r_1,\ldots, r_{m-1}\rangle.$$
Since $H_1$ is an $(m-1)$-prechain subgroup, $u_0(cl(J))$ is contained in $\supp(H_1)$.
By Lemma~\ref{for-minimal_ring_prop}, there exists $u_1\in [G,G]$ such that 
 $$u_1(K\cup u_0(cl(J)))\subset \supp(G_{\mathcal{R'}}).$$
 According to Lemma~\ref{minimal_chain_lem} (3), there exists $u_2\in [G_{\mathcal{R'}},G_{\mathcal{R'}}]$ such that 
$$u_2(u_1(K))\subset u_1u_0(J).$$
Then $u_0^{-1}u_1^{-1}u_2u_1\in [G,G]$ maps $K$ into $J$.

In the case of (2), $\partial_{+}K\in \left(\partial_{+}\supp(r_{m-1}), \partial_{-}K\right)$.
Therefore, the closure of $S^1\setminus K$ is contained in the support of a prechain subgroup
$$H_m:= \langle r_m, r_1,\ldots, r_{m-2}\rangle.$$
By Lemma~\ref{for-minimal_ring_prop}, there exists $v_1\in [G,G]$ such that
$$v_1(S^1\setminus K)\subset \supp(G_{\mathcal{R'}}).$$
Fix $r_i\in \mathcal{R'}$ and let
$$I:= \left[\partial_{-}\supp(r_i)+\varepsilon, \partial_{+}\supp(r_i)-\varepsilon\right]\subset \supp(G_{\mathcal{R'}})$$
where $\varepsilon$ is the one in $(\ref{epsilon_eq})$. 
We apply Lemma~\ref{minimal_chain_lem} (3) to $G_{\mathcal{R'}}$ and get $v_2\in [G_{\mathcal{R'}},G_{\mathcal{R'}}]$ such that 
$$v_2^{-1}(I)\subset v_1(S^1\setminus K).$$
It follows that 
$$v_2v_1(K)\subset S^1\setminus I,$$
and thus $v_2v_1(K)$ is contained in the support of the $(m-1)$-chain subgroup $H_m$.
Since $H_m$ is an $(m-1)$-prechain subgroup, $cl(u_0(J))$ is contained in $\supp(H_1)$.
By Lemma~\ref{for-minimal_ring_prop}, there exists $v_3\in [G,G]$ such that 
$$v_3\left(v_2v_1(K)\cup u_0(J)\right) \subset \supp(G_{\mathcal{R'}}).$$
We apply Lemma~\ref{minimal_chain_lem} (3) to $G_{\mathcal{R'}}$ and get $v_4\in [G_{\mathcal{R'}},G_{\mathcal{R'}}]$ such that 
$$v_4\left(v_3v_2v_1(K)\right)\subset v_3u_0(J).$$
Then $u_0^{-1}v_3^{-1}v_4v_3v_2v_1\in [G,G]$ maps $K$ into $J$.
\end{proof}

\begin{proof}[Proof of Theorem~\ref{simple_thm}]
Let $m\geq 3$.
Let $G=G_{\mathcal{R}}$ be an $m$-ring group with respect to $\mathcal{R}=\{r_i\}_{1\leq i\leq m}$
with a prechain subgroup (hence a chain subgroup) whose action on its support is minimal.

First, we apply Theorem~\ref{Higman1_lem} to $G$ and $S=\mathcal{R}=\{r_i\}_{1\leq i\leq m}$.
Let $s_1, s_2\in S\setminus \{1\}$ and $g\in G\setminus \{1\}$.
We show that there exists $u\in G$ such that $s_1,s_2,g,u$ satisfy $(\ref{Higman1_eq})$ in Theorem~\ref{Higman1_lem}.
Since $g\neq 1$, there exists an open interval $J$ such that $J\cap gJ=\emptyset$. 
According to Proposition~\ref{minimal_ring_prop}, there exists $u\in G$ such that 
$$u^{-1}\left(\supp(s_1)\cup \supp(s_2)\right)\subset J.$$
Since $J\cap gJ=\emptyset$, we have
$$u^{-1}\left(\supp(s_1)\cup \supp(s_2)\right)\cap gu^{-1}\left(\supp(s_1)\cup \supp(s_2)\right)=\emptyset,$$
and thus 
$$\left(\supp(s_1)\cup \supp(s_2)\right)\cap ugu^{-1}\left(\supp(s_1)\cup \supp(s_2)\right)=\emptyset,$$
which is $(\ref{Higman1_eq})$.
According to Theorem~\ref{Higman1_lem}, 
every normal subgroup $N\neq 1$ contains $[G,G]$.
In particular, $[G,G]$ equals to the second commutator subgroup of $G$.

Next, we apply Theorem~\ref{Higman1_lem} to $[G,G]$ and 
$$S'=\{g[r_i,r_{i+1}]g^{-1}\quad(1\leq i\leq m-1) \mid r_i\in \mathcal{R},\ g\in G\}.$$
We note that $[G,G]$ is normally generated by commutators of generators.
Since $\supp([r_i,r_{i+1}])\neq S^1$ for every $1\leq i\leq n-1$,
$\supp(s')\neq S^1$ for every $s'\in S'$.
Let $s'_1, s'_2\in S'\setminus \{1\}$ and $g'\in G\setminus \{1\}$.
We consider two cases: (1) $\supp(s'_1)\cup\supp(s'_2)\neq S^1$, and (2) $\supp(s'_1)\cup\supp(s'_2)= S^1$.
We first treat the case (1). 
When $\supp(s'_1)\cup\supp(s'_2)\neq S^1$, 
we show that there exists $u\in G$ such that $s'_1,s'_2,g',u$ satisfy $(\ref{Higman1_eq})$ in Theorem~\ref{Higman1_lem} by a similar argument as before.
We next treat the case (2). When $\supp(s'_1)\cup\supp(s'_2)= S^1$,  
we show that there exist $u_1, u_2\in G$ such that $s'_1,s'_2,g',u_1,u_2$ satisfy $(\ref{Higman2_eq})$ in Theorem~\ref{Higman1_lem}.
Since $g'\neq 1$, there exists an open interval $J$ such that $J\cap g'J=\emptyset$. 
According to Proposition~\ref{minimal_ring_prop}, the action of $[G,G]$ on $S^1$ is CO-transitive.
Therefore, there exists $u_1\in [G,G]$ such that 
$u_1^{-1}(\supp(s'_2))\subset J.$
Then
\begin{align*}
g'u_1^{-1}(\supp(s'_2))&\subset g'J\\
&\subset S^1\setminus J\\
&\subset S^1\setminus u_1^{-1}(\supp(s'_2))\\
&\subset u_1^{-1}(\supp(s'_1)),
\end{align*}
since $\supp(s'_1)\cup\supp(s'_2)= u_1^{-1}(\supp(s'_1)\cup\supp(s'_2))=S^1$.
It follows that 
\begin{align}\label{1.2eq}
u_1g'u_1^{-1}(\supp(s'_2))\subset \supp(s'_1).
\end{align}
Let 
$$s''_2=u_1g'u_1^{-1}(s'_2)u_1g'^{-1}u_1^{-1}.$$ 
That is, $s''_2$ is the conjugate of $s'_2$ by $u_1g'u_1^{-1}$.
Therefore, 
\begin{align}\label{1.2eq-2}
u_1g'u_1^{-1}\supp(s'_2)=\supp(s''_2).
\end{align}
By $(\ref{1.2eq})$,
$$\supp(s'_1)\cup \supp(s''_2)\subset \supp(s'_1)\neq S^1.$$
Replacing $s'_2$ with $s''_2$, the case (2) results in (1), and there exists $u\in G$ such that $s'_1,s''_2,g',u$ satisfy $(\ref{Higman1_eq})$ in Theorem~\ref{Higman1_lem}:
$$(\supp(s'_1)\cup \supp(s''_2))\cap ug'u^{-1}(\supp(s'_1)\cup \supp(s''_2))=\emptyset.$$
By $(\ref{1.2eq-2})$,
$$(\supp(s'_1)\cup u_1g'u_1^{-1}\supp(s'_2))\cap ug'u^{-1}(\supp(s'_1)\cup u_1g'u_1^{-1}\supp(s'_2)))=\emptyset,$$
which is $(\ref{Higman2_eq})$.
Therefore, in each of cases (1) and (2), we have confirmed a requirement for Theorem~\ref{Higman1_lem}.
By Theorem~\ref{Higman1_lem}, for every normal subgroup $N$ of $[G,G]$, 
$N$ contains the second commutator subgroup of $G$, which is equal to $[G,G]$.
It follows that $N=[G,G]$. In other words, $[G,G]$ is simple. 
\end{proof}

\subsection{Property A$\mathbb{R}$ of ring groups}\label{fixed_sec}
In this section, we show Theorem~\ref{AR_thm}.

An {\it $\mathbb{R}$-tree} is a non-empty metric space in which any two points are joined by a
unique arc which is isometric to a closed interval in the real line.
We say that a group $G$ has {\it property A$\mathbb{R}$} if every isometric action of $G$ on an $\mathbb{R}$-tree without global fixed points has an invariant line upon which $G$ acts by translation.

To show property A$\mathbb{R}$, we use a criterion for property A$\mathbb{R}$ in \cite{CV}.
To describe the criterion, we start with citing some terms from \cite{CV}.
Let $G$ be a finitely generated group and let $S$ be a finite generating set of $G$.
For $s_i\neq s_j\in S$, a {\it minipotent word} in $\{s_i, s_j\}$ is a word of the form 
\begin{align*}
s_i^{\varepsilon_1}s_j^{\varepsilon_2}\cdots s_i^{\varepsilon_{2l-1}}s_j^{\varepsilon_{2l}} \ \text{or}\ s_j^{\varepsilon_1}s_i^{\varepsilon_2}\cdots s_j^{\varepsilon_{2l-1}}s_i^{\varepsilon_{2l}},
\end{align*} 
where $\varepsilon_i\in \{\pm1\}$ for $1\leq i\leq 2l$.
Let $\Delta(G,S)$ be the graph with the vertex set $S$ and the edge set 
\begin{align*}
E=\left\{\{s_i,s_j\} \middle| 
\begin{array}{l}
\text{ $\exists$ a minipotent word $w_{i,j}$ in $\{s_i, s_j\}$}\\
\text{ which commutes with either $s_i$ or $s_j$.}
 \end{array}
 \right\}.
\end{align*}
For $s_i\neq s_j\in S$, we say $(s_i,s_j)\in S\times S$ is {\it distinguished} if $\exists k\in \mathbb{N}$ such that $[s_i, s_j^{(k)}]$ commutes with $s_j$,
where $[s_i, s_j^{(k)}]$ ($k\in \mathbb{N}$) are defined inductively by
\begin{align*}
&[s_i, s_j^{(1)}]:=[s_i,s_j],\quad(k=1),\\
&[s_i, s_j^{(k)}]:=\left[[s_i, s_j^{(k-1)}],s_j\right]\quad(k\geq 2).
\end{align*}
Let $\Delta'(G,S)$ be the subgraph of $\Delta(G,S)$ with vertex set $S$ and the edge set 
\begin{align*}
E'=\{\{s_i,s_j\}\mid \text{$(s_i,s_j)$ or $(s_j,s_i)$ is distinguished}\}.
\end{align*}
We note that if $s_i\neq s_j\in S$ commute, then both $(s_i,s_j)$ and $(s_j,s_i)$ are distinguished, and thus $\{s_i,s_j\}\in E'$.
Therefore, if a subset $V$ of $S$ consists of mutually commuting elements, then the subgraph of $\Delta'(G,S)$ spanned by $V$ is complete (in particular, connected).
A {\it conjugacy class} in $S$ is a non-empty subset of $S$ given as the intersection of $S$ and a conjugacy class in $G$.
A subset $V$ of $S$ is {\it $\Delta'$-connected} if the subgraph of $\Delta'(G,S)$ spanned by $V$ is connected.
A subset $V$ of $S$ is {\it dense} if for every $x\in S\setminus V$, there exist $v, v'\in V$ such that $(v,x)$ and $(x,v')$ are distinguished.
\begin{theorem}[{\cite[Theorem 2.3]{CV}}]\label{CV_thm}
Let $G$ be a group and let $S$ be a generating set of $G$.
Suppose that $\Delta(G,S)$ is the complete graph on $S$.
Suppose further that each conjugacy class in $S$ contains a $\Delta'$-connected dense subset.
Then $G$ has property A$\mathbb{R}$.
\end{theorem}

\begin{corollary}\label{CV_cor}
Let $G$ be a group and let $S$ be a generating set of $G$.
Suppose that for every $s,t\in S$, there exists a minipotent word in $\{s,t\}$ that represents the identity in $G$.
Suppose further that each conjugacy class in $S$ contains a subset $V$ consisting of mutually commuting elements, and such that for every $x\in S\setminus V$, there exists $v\in V$ which commutes with $x$. Then $G$ has property A$\mathbb{R}$.
\end{corollary}

The following is the key lemma.
\begin{lemma}\label{AR_lem}
Let $G_{\mathcal{R}}$ be a $m$-ring group with $\mathcal{R}=\{r_i\}_{1\leq i \leq m}$ such that
\begin{itemize} 
\item[(i)] $\partial_{+}\supp(r_i)=\partial_{-}\supp(r_{i+2}) \quad(1\leq i\leq m)$, 
\item[(ii)] $r_{i+1}r_i(\partial_{-}\supp(r_{i+1}))=\partial_{-}\supp(r_{i+2}) \quad(1\leq i\leq m-1)$, 
\end{itemize}
where indices are considered modulo $m$.
Then $G_{\mathcal{R}}$ has Property A$\mathbb{R}$.
\end{lemma}

\begin{proof}
In the following, the index $i$ is considered modulo $m$.
Let 
\begin{align}\label{r'_i}
r'_i=\left(r_{i+2}^2 r_{i+1}^2 r_i^2 r_{i-1}\right) r_i \left(r_{i+2}^2 r_{i+1}^2 r_i^2 r_{i-1}\right)^{-1} \quad(1\leq i\leq m).
\end{align}
Let 
$$S=\{r_i, r'_i\}_{1\leq i\leq m}.$$
We show that $G_{\mathcal{R}}$ and $S$ satisfy the requirements in Corollary~\ref{CV_cor}.

\begin{figure}
  \centering
  \includegraphics[width=0.7\textwidth]{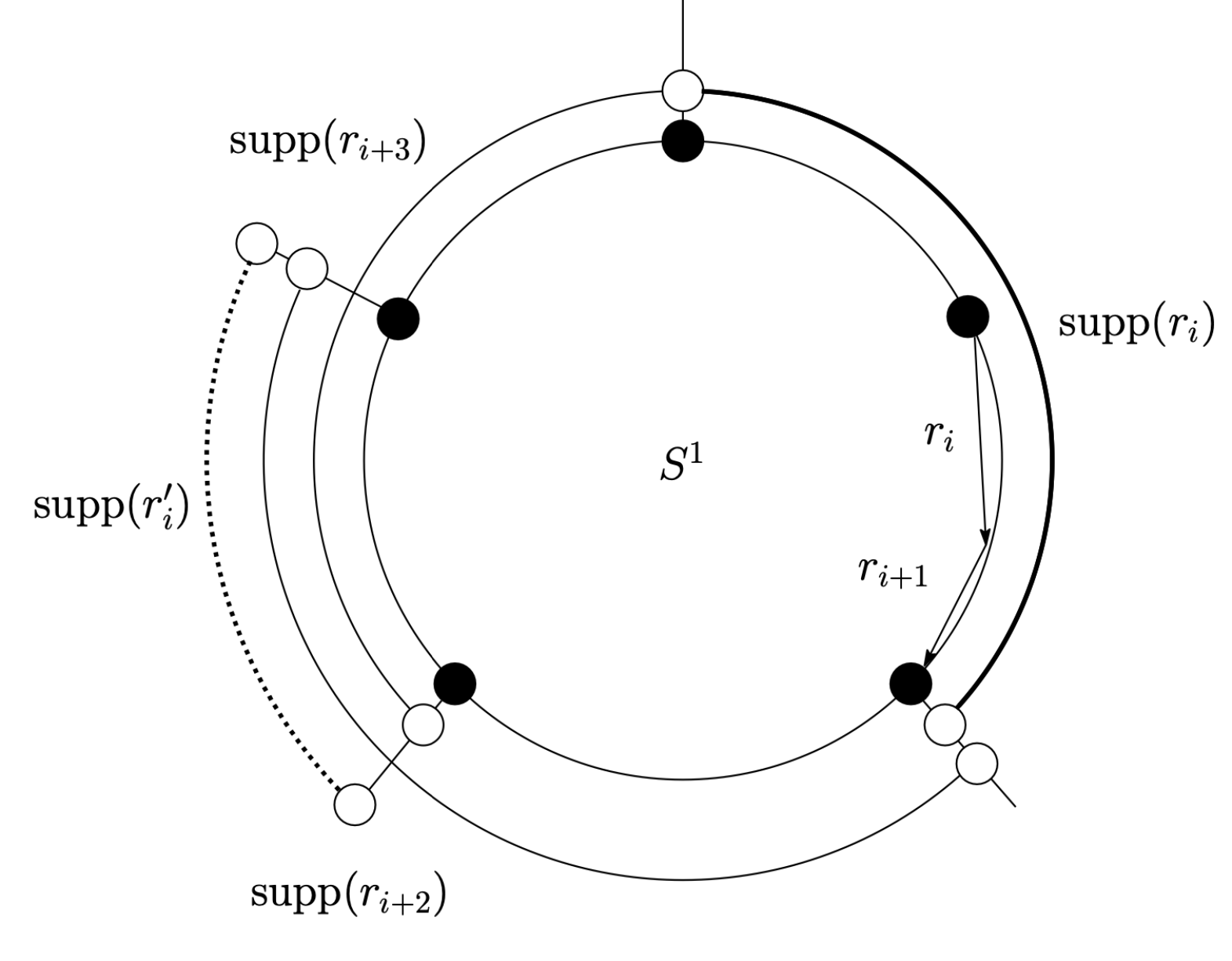}
  \caption{The case of $m=5$. The five points on the circle represent the endpoints of the supports of $r_1,\ldots, r_5$. The thick solid arc indicates $\supp(r_i)$, while the two thin solid arcs represent 
$\supp(r_{i+1})$ and $\supp(r_{i+2})$, respectively. 
The dashed arc indicates the region where 
$\supp(r'_i)$ is contained. The two arrows represent the actions of $r_i$ and $r_{i+1}$, respectively.}
 \label{5ring_fig}
\end{figure}

First, we show that 
\begin{align}\label{(i)}
\supp(r'_i)\subset \supp(r_{i+2})\cap \supp(r_{i+3}) \quad(1\leq i\leq m).
\end{align} 
See Figure~\ref{5ring_fig}.
For every $i$, 
according to $(\ref{r'_i})$ and the assumption (i),
\begin{align}
\begin{aligned}\label{(i)-2}
\partial_{+}\supp(r'_{i})&=r_{i+2}^2 r_{i+1}^2 r_i^2 r_{i-1}\left(\partial_{+}\supp(r_i)\right)\\
&=r_{i+2}^2 r_{i+1}^2\left(\partial_{+}\supp(r_i)\right)\\
&=r_{i+2}^2 r_{i+1}^2\left(\partial_{-}\supp(r_{i+2})\right)\\
&\in \left(r_{i+2} r_{i+1}\partial_{-}\supp(r_{i+2}), \partial_{+}\supp(r_{i+2})\right) \\
&= \left(\partial_{-}\supp(r_{i+3}), \partial_{+}\supp(r_{i+2})\right).
\end{aligned}
\end{align}
Here, note that $\partial_{+}\supp(r_i)\notin \supp(r_{i-1})\cup \supp(r_i)$.
According to $(\ref{r'_i})$ and the assumption (ii),
\begin{align}
\begin{aligned}\label{(i)-1}
\partial_{-}\supp(r'_{i})&=r_{i+2}^2 r_{i+1}^2 r_i^2 r_{i-1}\left(\partial_{-}\supp(r_i)\right)\\
&=r_{i+2} (r_{i+2} r_{i+1}) (r_{i+1} r_i) (r_i r_{i-1})\left(\partial_{-}\supp(r_{i})\right)\\
&=r_{i+2} (r_{i+2} r_{i+1}) (r_{i+1} r_i) \left(\partial_{-}\supp(r_{i+1})\right) \\
&= r_{i+2} (r_{i+2} r_{i+1})\left(\partial_{-}\supp(r_{i+2})\right) \\
&= r_{i+2} \left(\partial_{-}\supp(r_{i+3})\right) \\
&\in \left(\partial_{-}\supp(r_{i+3}), \partial_{+}\supp(r_{i+2})\right).
\end{aligned}
\end{align}
$(\ref{(i)})$ follows from $(\ref{(i)-2})$ and $(\ref{(i)-1})$. According to $(\ref{(i)})$,
\begin{align*}
[r'_i, r]=1 \quad(r\in S\setminus \{r_{i+2}, r_{i+3}\}).
\end{align*}

Second, we show that each $r_i$ satisfies the following:
\begin{align}\label{(ii)}
\supp(r'_i) \cap r_{i+2}^{-1} \left(\supp(r'_i)\right)=\emptyset. 
\end{align}
By $(\ref{(i)-2})$ and  $r_{i+1}^2(\partial_{-}\supp(r_{i+2})) \in \supp(r_{i+1})$, 
\begin{align}
\begin{aligned}\label{(ii)-1}
r_{i+2}^{-1} \left(\supp(r'_i)\right) &=\left(r_{i+2}^{-1} (\partial_{-}\supp(r'_{i})), r_{i+2}^{-1} (\partial_{+}\supp(r'_{i}))\right)\\
&=\left(r_{i+2}^{-1} (\partial_{-}\supp(r'_{i})), r_{i+2}r_{i+1}^2 (\partial_{-}\supp(r_{i+2}))\right)\\
&\subset \left(r_{i+2}^{-1} (\partial_{-}\supp(r'_{i})), r_{i+2} (\partial_{+}\supp(r_{i+1}))\right)\\
&=\left(r_{i+2}^{-1} (\partial_{-}\supp(r'_{i})), r_{i+2} (\partial_{-}\supp(r_{i+3}))\right).
\end{aligned}
\end{align}
$(\ref{(ii)})$ follows from $(\ref{(i)-1})$ and $(\ref{(ii)-1})$.
According to $(\ref{(ii)})$,
\begin{align}\label{(ii)'}
[r_{i+2}^{-1}r'_i r_{i+2}, r'_i]=1.
\end{align}

Third, we show that each $r_i$ satisfies the following:
\begin{align}\label{(iii)}
\supp(r'_i) \cap r_{i+3} \left(\supp(r'_i)\right)=\emptyset. 
\end{align}
By $(\ref{(i)-1})$, (i) and (ii),
\begin{align}
\begin{aligned}\label{(iii)-1}
\partial_{-}r_{i+3} \left(\supp(r'_i)\right) &=r_{i+3} \left(\partial_{-}\supp(r'_i)\right) \\
&= r_{i+3}r_{i+2}(\partial_{-}\supp(r_{i+3}))=\partial_{-}\supp(r_{i+4})\\
&=\partial_{+}\supp(r_{i+2}).
\end{aligned}
\end{align}
$(\ref{(iii)})$ follows from $(\ref{(i)})$ and $(\ref{(iii)-1})$.
According to $(\ref{(iii)})$,
\begin{align}\label{(iii)'}
[r_{i+3}^{-1}r'_i r_{i+3}, r'_i]=1.
\end{align}

Finally, we confirm the requirements in Corollary~\ref{CV_cor} for $G_{\mathcal{R}}$ and $S$.
First, we show that for every $s,t\in S$, there exists a minipotent word in $\{s,t\}$ that represents the identity in $G$.
We first consider the case where $s=r_i$ and $t=r_j$. If $|i-j|>1$, $r_i$ commutes with $r_j$, according to the requirements for ring groups.
If $j=i+1$, by the assumption (ii) and Lemma~\ref{2chain_lem},
\begin{align*}
[r_{i}, \left(r_{i+1}r_i\right)r_{i+1}\left(r_{i+1}r_i\right)^{-1}]=1,
\end{align*}
where $[r_{i}, \left(r_{i+1}r_i\right)r_{i+1}\left(r_{i+1}r_i\right)^{-1}]$ is a minipotent word in $\{r_{i}, r_{i+1}\}$.
We next consider the case where $s=r'_i$.  
According to $(\ref{(i)})$, $r'_i$ commutes with every $r\in S\setminus \{r_{i+2}, r_{i+3}\}$. 
According to $(\ref{(ii)'})$ (resp. $(\ref{(iii)'})$), there exists a minipotent word in $\{r'_i, r_{i+2}\}$ (resp.\ $\{r'_i, r_{i+3}\}$) that represents the identity in $G$.
We note that $[r_{i+2}^{-1}r'_i r_{i+2}, r'_i]$ and $[r_{i+3}^{-1}r'_i r_{i+3}, r'_i]$ are minipotent words in $\{r'_{i}, r_{i+2}\}$ and $\{r'_{i}, r_{i+3}\}$, respectively.
Next, we show that each conjugacy class in $S$ contains a subset $V$ consisting of mutually commuting elements, and such that for every $x\in S\setminus V$, there exists $v\in V$ which commutes with $x$. 
For every $1\leq i\leq m$, let $V_i=\{r_i,r'_i\}\subset S.$
Since $r'_i$ is defined as a conjugate of $r_i$, for each conjugacy class in $S$, there exists $i$ such that $V_i\subset S$.
According to $(\ref{(i)})$, each $V_i$ consists of mutually commuting elements. 
Let $x\in S\setminus \{r_i, r'_{i}\}$. 
If $x\neq r_{i-1}$ or $r_{i+1}$, $x$ commutes with $r_i$.
If $x= r_{i-1}$ or $r_{i+1}$, $x$ commutes with $r'_i$, according to $(\ref{(i)})$.
Therefore, for every $x\in S\setminus V$, there exists $v\in V$ which commutes with $x$. 
Applying Corollary~\ref{CV_cor} for $G_{\mathcal{R}}$ and $S$, $G_{\mathcal{R}}$ has property A$\mathbb{R}$.
\end{proof}

\begin{theorem}[{\cite[Theorem 4.5, see also Lemma 4.2]{KKL}}]\label{embedding_thm}
Let $n\in \mathbb{N}$.
Let $H$ be an $n$-generated subgroup of $\mathrm{Homeo}^{+}(\mathbb{R})$.
We define $a,b\in \mathrm{Homeo}^{+}(\mathbb{R})$ by
\begin{align*}
&a(x)=x+1, \\
&b(x)=
\begin{cases}
x &(x\leq 0)\\
2x &(0<x<1)\\
x+1 &(x\geq 1).
\end{cases}
\end{align*}
Then, $\langle H,a,b \rangle$ is an $(n+2)$-chain group $G_{\mathcal{F}}$ acting minimally on its support with $\mathcal{F}=\{f_i\}_{1\leq i \leq n+2}$ such that
\begin{itemize} 
\item[(i)] $\partial_{+}\supp(f_i)=\partial_{-}\supp(f_{i+2}) \quad(1\leq i\leq n)$, 
\item[(ii)] $f_{i+1}f_i(\partial_{-}\supp(f_{i+1}))=\partial_{-}\supp(f_{i+2}) \quad(1\leq i\leq n)$. 
\end{itemize}
\end{theorem}

\begin{proof}[Proof of Theorem~\ref{AR_thm}]
Let $H$ be a $2$-generated subgroup of $\mathrm{Homeo}^+(\mathbb{R})$.
By Theorem~\ref{embedding_thm}, there exists a $4$-chain group $G_{\mathcal{F}}$ acting minimally on its support such that 
\begin{itemize} 
\item[(i)] $\partial_{+}\supp(f_i)=\partial_{-}\supp(f_{i+2}) \quad(1\leq i\leq 2)$, 
\item[(ii)] $f_{i+1}f_i(\partial_{-}\supp(f_{i+1}))=\partial_{-}\supp(f_{i+2}) \quad(1\leq i\leq 2)$
\end{itemize}
and that $H$ embeds into $G_{\mathcal{F}}$.

We fix a basepoint $p$ on $S^1$ and an orientation preserving homeomorphism 
$$\phi:\left(\partial_{-}\supp(f_{1}),\partial_{+}\supp(f_{4})\right)\to S^1\setminus \{p\}.$$
Let
$$r_i=\phi\circ f_i\circ \phi^{-1} \quad(1\leq i\leq 4).$$
Note that
\begin{align*}
\supp(r_i)=
\begin{cases}
(p,\phi(\partial_{-}\supp(f_{3}))) &(i=1)\\
(\phi(\partial_{-}\supp(f_{i})),\phi(\partial_{+}\supp(f_{i}))) &(i=2,3)\\
(\phi(\partial_{+}\supp(f_{2})), p) &(i=4).
\end{cases}
\end{align*}
Take $r_5\in \mathrm{Homeo}^+(S^1)$ such that 
$\supp(r_5)=(\phi(\partial_{+}\supp(f_{3})),\phi(\partial_{-}\supp(f_{2})))$, $r_5r_4(\phi(\partial_{-}\supp(f_{4})))=p$ and $r_1r_5(p)=\phi(\partial_{-}\supp(f_{2}))$.
Let  $G$ be a $5$-ring group generated by a sequence $\{r_1, r_2, r_3, r_4, r_5\}$ satisfying requirements in Lemma~\ref{AR_lem}.
Thus $G$ is a group with property A$\mathbb{R}$, containing $H$ as a subgroup of a $4$-chain subgroup which acts minimally on its support.
\end{proof}

\subsection{Uncountably many finitely generated simple subgroups of $\mathrm{Homeo}^{+}(S^1)$}\label{uncountable_sec}
In this section, we show Corollary~\ref{uncountable2_cor}.

\begin{theorem}[{\cite[III.C.40]{Harpe}, \cite[Lemma 5.2]{KKL}}]\label{uncountable_thm}
There exists a set $\mathcal{N}$ of $2$-generated subgroups of $\mathrm{Homeo}^{+}([0,1/2])$,
consisting of uncountably many isomorphism types of groups.
\end{theorem}

\begin{lemma}[{\cite[Remark 2.6]{CV}}]\label{fg_lem}
If a group has property A$\mathbb{R}$, then the commutator subgroup is finitely generated.
\end{lemma}

\begin{proof}[Proof of Corollary~\ref{uncountable2_cor}]
Let $\mathcal{N}$ be the set of $2$-generated subgroups of $\mathrm{Homeo}^{+}([0,1/2])$ in Theorem~\ref{uncountable_thm}.

First, we construct a ring group $R_N$ for each $N\in \mathcal{N}$ such that 
$N$ embeds into a chain subgroup of $R_N$ which acts minimally on its support, and
$R_N$ has property A$\mathbb{R}$.
We fix $N\in \mathcal{N}$.
By Theorem~\ref{embedding_thm}, 
$\langle N,a,b \rangle$ is a $4$-chain group $G_{\mathcal{F}}$ acting minimally on its support with $\mathcal{F}=\{f_i\}_{1\leq i \leq 4}$ such that
\begin{itemize} 
\item[(i)] $\partial_{+}\supp(f_i)=\partial_{-}\supp(f_{i+2}) \quad(1\leq i\leq 2)$, 
\item[(ii)] $f_{i+1}f_i(\partial_{-}\supp(f_{i+1}))=\partial_{-}\supp(f_{i+2}) \quad(1\leq i\leq 2)$.
\end{itemize}
A similar argument as in the proof of Theorem~\ref{AR_thm} shows that $G_{\mathcal{F}}$ embeds in a $5$-ring group $R_N$ with property A$\mathbb{R}$.

Next, we show that $N$ embeds into the commutator subgroup of $G_{\mathcal{F}}$. 
Let $n_1,n_2$ be generators of $N$. 
Since $n_1$ is the identity outside of $[0,1/2]$, the conjugate of $n^{-1}$ by $a$ is the identity outside of $a[0,1/2]=[1,2/3]$.
Therefore, $n_1(an_1^{-1}a^{-1})=[n_1,a]$ acts as $n_1$ on $[0,1/2]$ and as a copy of $n_1^{-1}$ on $[1,3/2]$.
Similarly, $[n_2,a^{-1}]$ acts as $n_2$ on $[0,1/2]$ and as a copy of $n_2^{-1}$ on $[-1,-1/2]$.
With this observation,
$\langle [n_1,a],[n_2,a^{-1}] \rangle\subset [G_{\mathcal{F}}, G_{\mathcal{F}}]$ is isomorphic to
$N=\langle n_1,n_2 \rangle$.

Finally, we show that the set of commutator subgroups of $R_N$ consists of uncountably many finitely generated simple groups.
Since there exist at most countably finitely generated subgroups for a countable group, 
$\{[R_N,R_N]\}_{N\in \mathcal{N}}$ contains uncountably many groups in terms of group isomorphism.
According to Theorem~\ref{simple_thm}, every $[G_N,G_N]$ is simple.
According to Lemma~\ref{fg_lem}, every $[G_N,G_N]$ is finitely generated.
\end{proof}


\begin{thebibliography}{99}
\bibitem{BBK}
C.\ Bleak, M.\ Brin, M.\ Kassabov, J.\ T.\ Moore and M.\ C.\ B.\
Zaremskym, \textit{Groups of fast homeomorphisms of the interval and the ping-pong argument}, J.\
Comb.\ Algebra, {\bf 3}(1), 1--40, 2019.
  \bibitem{Brown}
    K.\ S.\ Brown, \textit{Finiteness properties of groups}, J.\ Pure Appl.\ Algebra, {\bf 44} (1-3), 45--75, 1987.
  \bibitem{CFP}
  J.\ W.\ Cannon, W.\ J.\ Floyd, and W.\ R.\ Parry, \textit{Introductory notes on Richard Thompson's groups}, Enseign.\ Math.\ (2) {\bf 42}, 215--256, 1996.
  \bibitem{CV}
  M.\ Culler, K.\ Vogtmann, \textit{A group-theoretic criterion for property FA}, Proc.\ Amer.\ Math.\ Soc. {\bf 124}, 677--683, 1996.
  \bibitem{Harpe} P. de la Harpe, \textit{Topics in geometric group theory}, Chicago Lectures in Mathematics, University of Chicago Press, Chicago, IL, 2000.
   \bibitem{Higman} G.\ Higman, On infinite simple permutation groups, Publ. Math. Debrecen 3 (1954), 221--226 (1955).
  \bibitem{HL}
  J.\ Hyde and Y.\ Lodha, \textit{Finitely generated infinite simple groups of homeomorphisms of the real line}, Invent.\ math. {\bf 218}, 83--112, 2019. 
 \bibitem{HLR}
  J.\ Hyde, Y.\ Lodha and C.\ Rivas,
\textit{Two new families of finitely generated simple groups of homeomorphisms of the real line}, Journal of Algebra, {\bf 635}, 1--22,
2023.
    \bibitem{Kato2023}  M.\ Kato, \textit{Semi-simple actions of the Higman-Thompson groups Tn on finite-dimensional CAT(0) spaces}, Geom.\ Dedicata {\bf 217}(2023), no.5, Paper No. 85, 20 pp.
  \bibitem{KKL}
   S.\ Kim, T.\ Koberda and Y.\ Lodha,
   \textit{Chain groups of homeomorphisms of the interval and the circle}, Ann.\ Sci.\ de l'ENS (4) {\bf 52}, 797--820, 2019.
    \end{thebibliography}
\end{document}